\documentclass {article}%[11pt,oneside]{article}
\usepackage[english]{babel}
\usepackage{epsfig}
\usepackage{graphics}
\usepackage{graphicx}
\usepackage{amssymb}
\usepackage{amsmath}
\usepackage{amsthm}
\usepackage{array}
\usepackage[latin1]{inputenc}
\usepackage{fancyhdr}
\usepackage{tikz-cd}

\newtheorem{teo}{Theorem}[section]

\newtheorem{pr}[teo]{Proposition}
\newtheorem{lem}[teo]{Lemma}

\newtheorem{cnj}[teo]{Conjecture}
\newtheorem{ob}[teo]{Remark}
\newtheorem{al}[teo]{Algorithm}

\title{On a parameterization of $(1,1)$-knots}
\author{Jos\'e Fr\'ias }

\begin{document}
\maketitle

%\section{Introduction}
%\blindtext
%\end{document}
%\begin{document}

\begin{abstract}
A $(1,1)$-knot in the 3-sphere is a knot that admits a 1-bridge presentation with respect to a Heegaard torus in $\mathbb{S}^{3}$. A new parameterization of $(1,1)$-knots distinct from the classical ones is introduced. This parameterization is obtained from minimal-length representatives of homotopy classes of arcs in the mutipunctured plane. In the particular case of satellite $(1,1)$-knots, it is proven that the introduced parameterization is essentially unique. A generalization of this parameterization to the family of $(g,1)$-knots for any $g\geq 1$ is proposed.      
\end{abstract}

\section{Introduction}\label{sec:1}

A knot $K\subset\mathbb{S}^{3}$ is called a \emph{$(g,b)$-knot} if there exists a genus-$g$ Heegaard splitting of the 3-sphere, $\mathbb{S}^{3}=H_{1}\cup H_{2}$, such that $K\cap H_{i}$ is the union of $b$ mutually disjoint properly embedded trivial arcs, $i=1,2$. In the present work, we are interested in the family of $(1,1)$-knots. This family of knots contains the very well known subfamilies of torus knots and rational knots, and it is contained in the family of knots with tunnel number 1.     

Representations of knots in bridge positions with respect to Heegaard surfaces may be helpful in the study of particular surfaces concerning the knots  (see \cite{CMS} or \cite{EM3}). There are previously known parameterizations of $(1,1)$-knots, such as Schubert and Conway normal forms (see \cite{CK}). The Schubert normal form requires a $4$-tuple of integers to parameterize a given $(1,1)$-knot. On the contrary, the parameterization  proposed in this work has an unbounded number of parameters. In \cite{CM}, the authors studied an algebraic representation of $(1,1)$-knots via the mapping class group of the twice punctured torus $MCG_{2}(T)$.

The parameterization of $(1,1)$-knots that we propose has a geometric motivation. We establish a relation between a $(1,1)$-knot $K$ in a specific position and an arc $\beta$ in the $\epsilon$-multipunctured plane $\mathcal{B}_{\epsilon}$, such that $\partial \beta \subset \partial \mathcal{B}_{\epsilon} $. There is a unique minimal-length representative $\beta_{0}$ in the homotopy class of $\beta$ in $\mathcal{B}_{\epsilon}$. A parameterization of the arc $\beta_{0}$ induces the parameterization of $K$ as shown in Theorem \ref{teo:31}, we name it a tight parameterization of $K$. It would be an interesting topic the study of the relation between this representation of $(1,1)$-knots and those mentioned in the previous paragraph. 

In Section \ref{sec:2}, we analyze minimal-length arcs in the multipunctured plane $\mathcal{B}_{\epsilon}$ with one of its endpoints in a fix component of $\partial \mathcal{B}_{\epsilon}$. We define two simplifications of the curve $\beta_{0}$ obtained by decreasing the value of $\epsilon$. The connection between a $(1,1)$-knot and a minimal-length curve in $\mathcal{B}_{\epsilon}$ that induces the parameterization of the knot is established in Section \ref{sec:3}. It is proven in Section \ref{sec:4} that in the family of satellite $(1,1)$-knots, the tight parameterization of a knot is essentially unique (Theorem \ref{teo:up}). An algorithm to find  tight parameterizations for satellite $(1,1)$-knots based on the description of these knots by Morimoto and Sakuma  is presented (Algorithm \ref{al:1}). Finally, we propose in Section \ref{sec:5} a generalization of Theorem \ref{teo:31}  to the general case of $(g,1)$-knots for any $g\geq 1$. To this end, we consider the model of the hyperbolic geoboard (arcs embedded in a hyperbolic multipunctured disk), and suggest how the results obtained in the case $g=1$ could be extended.

%\section{Preliminary results}\label{sec:2}

\section{The multipunctured plane}\label{sec:2}

In this section, we introduce a model that will be useful to establish the proposed parameterization of $(1,1)$-knots. Consider the plane $\mathbb{R}^{2}$ equipped with the flat metric and the standard unitary square tiling $\mathcal{T}$ with vertices at the points in the plane with integer coordinates. Let $W$ be the set of points in the plane with coordinates $(l/2,m/2)$, where $l$ and $m$ are odd integers. For a sufficiently small real number $1/2>\epsilon > 0$, consider the set $\mathfrak{B}_{\epsilon}=\mathbb{R}^{2}\setminus \bigcup D_{\epsilon}(w)$, where $D_{\epsilon}(w)$ is an $\epsilon$-radius open disk centered at $w$ for every $w\in W$. We will call the set $\mathfrak{B}_{\epsilon}$ the \emph{$\epsilon$-multipunctured plane} and it is the plane with small disks centered at the midpoints of the tiles in $\mathcal{T}$ removed.  \par
Let $\beta$ be a smooth curve in $\mathfrak{B}_{\epsilon}$ with endpoints $z_{0}$ and $z_{1}$ in $\partial D_{\epsilon}(w_{0})$ and $\partial D_{\epsilon}(w_{0}')$, respectively, for some points $w_{0}, w_{0}'\in W$ (it could be $w_{0}=w_{0}'$). Suppose $\beta$ is oriented from $z_{0}$ to $z_{1}$. To establish a framework, we can stick the arc $\beta$ by taking $w_{0}$ to be a fixed point in $W$, say $w_{0}=(1/2,1/2)$. We are interested  in the homotopy class of $\beta$ in $\mathfrak{B}_{\epsilon}$ of arcs with endpoints in $\partial D_{\epsilon}(w_{0})$ and $\partial D_{\epsilon}(w_{0}')$.  \par

The problem of finding shortest  homotopic paths in a metric space under topological constraints
is one of the classical problems in geometric optimization. In the proof of Lemma 1 from \cite{ABM}, the authors show  that there exists a unique free loop of shortest length in any homotopy class of closed curves in a multipunctured plane. In our particular case,  this implies that there is a unique minimal length curve $\beta_{0}$ within the homotopy class of $\beta$ in  $\mathfrak{B}_{\epsilon}$ as a curve with endpoints in $\partial D_{\epsilon}(w_{0})$ and $\partial D_{\epsilon}(w_{0}')$.  Intuitively, imagine $\beta$ is represented by a thin physical string on the \emph{geoboard} (physical board with  nails pinned to the vertices of a square tiling), such that the endpoints of the string are tied to two nails. Once the string is completely tightened on the geoboard, we get a representation of the minimal-length curve $\beta_{0}$ in the homotopy class of $\beta$ in $\mathfrak{B}_{\epsilon}$ (see Figure \ref{fig:b2}). Moreover, the curve  $\beta_{0}$ decomposes as $\beta_{0}=\gamma_{1}\cup\delta_{1}\cup\gamma_{2}\cup\dots\cup\delta_{n}\cup\gamma_{n+1}$, where $\delta_{i}$ is  a point in $\partial D_{\epsilon}(w_{i})$ or a monotonous curve contained in $\partial D_{\epsilon}(w_{i})$ for some $w_{i}\in W$ (there is a smooth parameterization of the curve whose derivative never vanishes), while $\gamma_{j}\subset \mathfrak{B}_{\epsilon}$ is a straight line segment with interior disjoint from $\partial \mathfrak{B}_{\epsilon}$, sharing  endpoints with $\delta_{j-1}$ and $\delta_{j}$, and touching $\partial D_{\epsilon}(w_{i-1})$ and $\partial D_{\epsilon}(w_{i})$ in a tangent direction (except for the start point of $\gamma_{1}$ and the end point of $\gamma_{n+1}$), for every $i$ and $j$. Note that these minimal-length curves are related to the classical problem of the Dubin paths, which are commonly used in the fields of robotics and control theory (see \cite{BCL} or \cite{DL}).  \par
\begin{figure}
  \centering
    \includegraphics{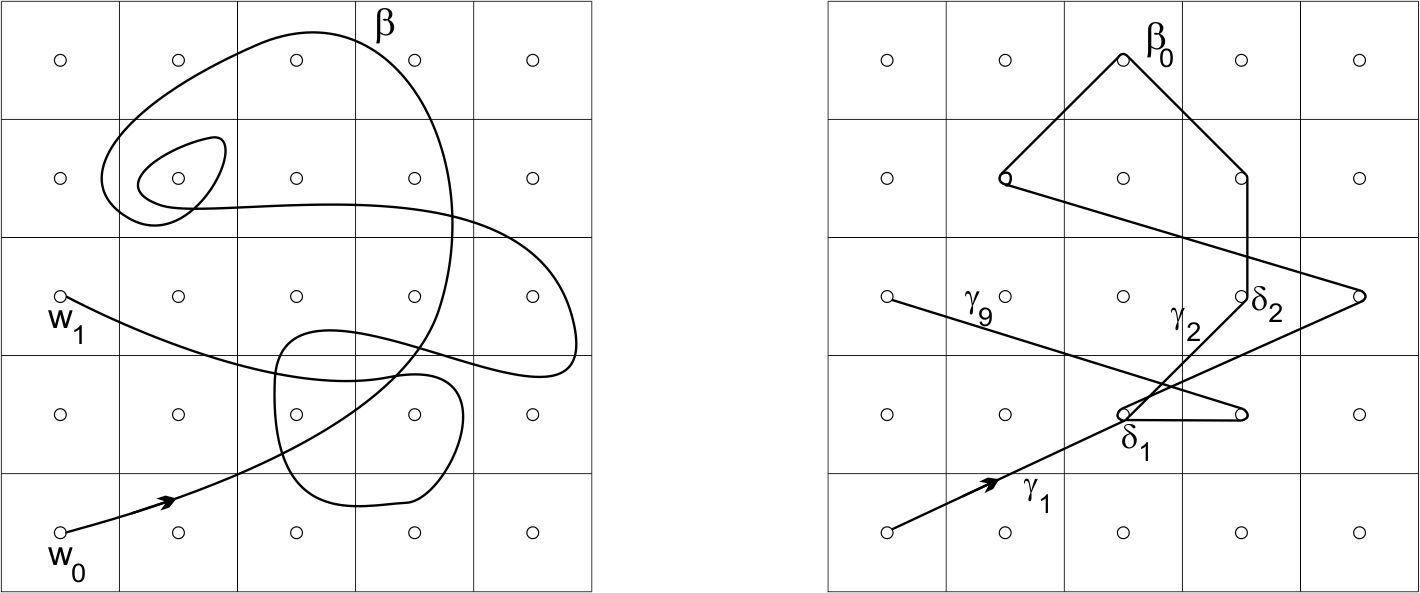}%[width=8cm]{b2}
  \caption{Homotopy between $\beta$ and its minimal-length homotopic curve }
  \label{fig:b2}
\end{figure}

We aim to parameterize all the homotopy classes of smooth curves in $\mathfrak{B}_{\epsilon}$ starting at $\partial D_{\epsilon}(w_{0})$ and  ending at any component of $\partial\mathfrak{B}_{\epsilon}$. Let $\beta\subset \mathfrak{B}_{\epsilon}$  be a smooth  curve as in the previous paragraph and let $\beta_{0}$ be its minimal-length homotopic curve. We describe how to simplify the curve $\beta_{0}$ to a canonical representative in the homotopy class of $\beta$ by decreasing the magnitude of $\epsilon$ and, consequently, extending the space $\mathfrak{B_{\epsilon}}$ (to be more precise, we extend the curves as we extend the space).  %The curve $\beta_{0}$ may not be canonical in the homotopy class of $\beta$, however it can simplified to a canonical representative by decreasing the magnitude of $\epsilon$ and, consequently, extending the space $\mathfrak{B_{\epsilon}}$. 

Suppose that $\delta_{i-1}\cup \gamma_{i}\cup \delta_{i}\cup\gamma_{i+1}\cup \delta_{i+1}$ is a subcurve of $\beta_{0}$ as previously described, where the curves $\delta_{i-1}$ and $\delta_{i+1}$ are winding around the distinct points $w_{i-1}, w_{i+1}\in W$ in opposite directions (one counterclockwise and the other clockwise), while the arc $\delta_{i}$ covers an angle smaller than $\pi$ around $w_{i}$ in any direction. Let  $\lambda$ be the straight line segment in the plane  connecting and oriented from the point $w_{i-1}$ to $w_{i+1}$. Suppose that the point $w_{i}$ is on the same side of the arcs  $\lambda$ and $\gamma_{i}\cup \delta_{i}\cup\gamma_{i+1}$ as we move in the direction of their orientations, as shown in the left-hand picture of Figure \ref{fig:b3}. It follows from an elementary geometric argument that there exists a positive number $\epsilon'<\epsilon$ such that if we consider the $\epsilon'$-punctured plane  $\mathfrak{B}_{\epsilon'}$, then the subcurve $\delta_{i-1}\cup \gamma_{i}\cup \delta_{i}\cup\gamma_{i+1}\cup \delta_{i+1}$ of  $\beta_{0}$ gets simplified to a subcurve $\delta_{i-1}'\cup \gamma_{i}'\cup \delta_{i}'$ in $\beta_{0}'$, the  minimal-length representative in the homotopy class of the extension of $\beta$ to $\mathfrak{B}_{\epsilon'}$. In the right-hand picture in Figure \ref{fig:b3} we exemplify how this simplification looks like and we shall call it an \emph{arc reduction} of $\beta_{0}$. \par
\begin{figure}
  \centering
    \includegraphics{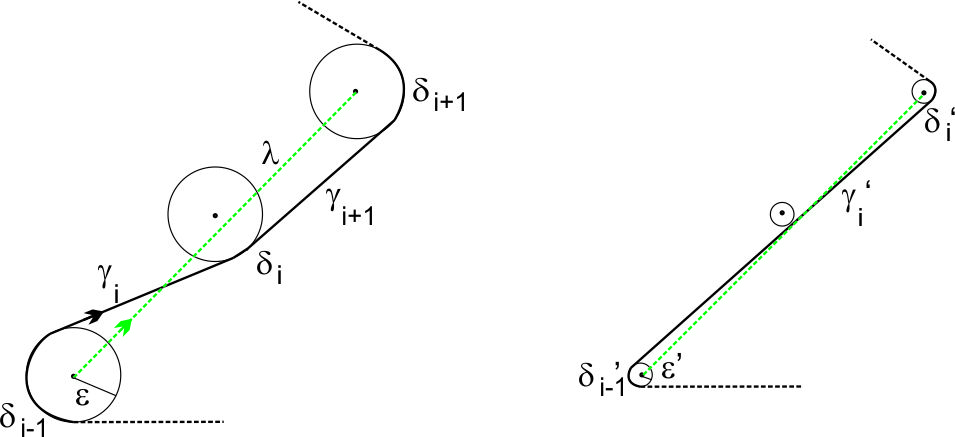}%[width=8cm]{b2}
  \caption{Arc reduction of the curve $\beta_{0}$}
  \label{fig:b3}
\end{figure} 

Now we continue with another technical simplification related to that of the previous paragraph. Let $\epsilon$, $\mathfrak{B}_{\epsilon}$ and $\beta$ be as before and let $\beta_{0}$ be the shortest curve in the homotopy class of $\beta$ in $\mathfrak{B}_{\epsilon}$ such that $\beta_{0}$ does not admit an arc reduction. Let $\delta_{i-1}\cup\gamma_{i}\cup \delta_{i}\cup \gamma_{i+1}\cup\delta_{i+1}$ be a subcurve of $\beta_{0}$ such that $\delta_{i}$ covers an angle of $r$ radians around $w_{i}\in W$, and the curves $\delta_{i-1}$ and $\delta_{i+1}$ wind around the points $w_{i-1}, w_{i+1}\in W$, respectively. Let $\lambda_{i}$ and $\lambda_{i+1}$ be the straight line segments connecting $w_{i-1}$ with $w_{i}$ and  $w_{i}$ with $w_{i+1}$ (see Figure \ref{fig:b4}). By taking a value $\epsilon'<\epsilon$ to define the space $\mathfrak{B}_{\epsilon'}$, the segments $\gamma_{i}'$ and $\gamma_{i+1}'$, corresponding to $\gamma_{i}$ and $\gamma_{i+1}$ in the shortest path in the homotopy class of $\beta$ in $\mathfrak{B}_{\epsilon'}$, approaches to $\lambda_{i}$ and $\lambda_{i+1}$, respectively. Consequently, the angle $r'$ covered by the corresponding arc $\delta_{i}'$ may decrease (this angle remains the same if the curves $\delta_{i-1}$, $\delta_{i}$ and $\delta_{i+1}$ turn in the same direction and gets reduced in any other case). In the limit, we have an angle $r_{0}$ which is delimited by the points of tangency of parallel lines to $\lambda_{i}$ and $\lambda_{i+1}$ on $\partial D_{\epsilon}(w_{i})$ as shown in Figure \ref{fig:b4}. Suppose that the angle $r_{0}$ satisfies  $(m-1)\pi<|r_{0}|\leq m\pi$, for some integer $m\geq 0$, then there exists $\epsilon'\leq\epsilon$ such that the angle $r'$ covered by the corresponding  curve $\delta_{i}'$ on $\partial D_{\epsilon'}(w_{i})$   satisfies $|r_{0}|\leq |r'| \leq m\pi$. If this last condition is satisfied, we shall say that the curve $\beta_{0}'$ is \emph{stabilized} at $\delta_{i}'$.\par

\begin{figure}
  \centering
    \includegraphics{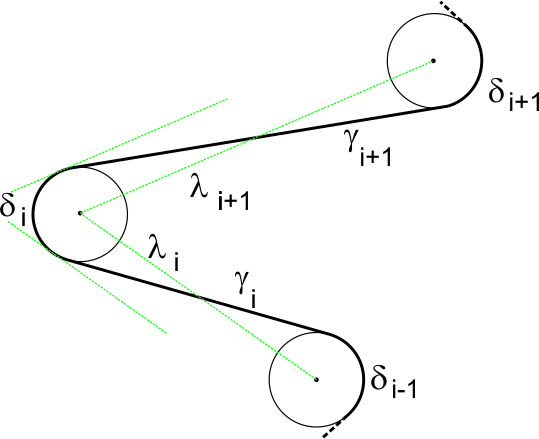}%[width=8cm]{b2}
  \caption{Stabilization at $\delta_{i}$}
  \label{fig:b4}
\end{figure} 

If $\epsilon>0$ is chosen such that in $\mathfrak{B}_{\epsilon}$ the curve $\beta_{0}=\gamma_{1}\cup\delta_{1}\cup\gamma_{2}\cup\dots\cup\delta_{n}\cup\gamma_{n+1}$, which  is the minimal-length representative in the homotopy class of $\beta$,  does not admit any arc reduction and it is stabilized at $\delta_{i}$, $i=1,\dots , n$, we say that $\beta_{0}$ is \emph{simplified}.\par  %We say a reduced minimal-length curve in the homotopy class of $\beta$ is canonical since it is not possible to find a minimal-length curve with a smaller number of components by decreasing the value of $\epsilon$. \par

\begin{pr}\label{teo:p1} Let $\beta\subset \mathfrak{B}_\epsilon$ an oriented smooth curve starting at $\partial D_{\epsilon}(w_{0})$ and ending at any component of $\partial \mathfrak{B}_{\epsilon}$. The homotopy class of $\beta$ is parameterized by a finite sequence of integers $(p_{1}, q_{1}, m_{1}, p_{2}, q_{2},\ldots,m_{n},p_{n+1},q_{n+1})$, for some $n\geq 0$, where $p_{i}$ or $q_{i}$ can be zero but not at the same time. % and $m_{i}\neq 0$, $i=1,\dots,n$.

\end{pr}

\begin{proof}

This parameterization follows from taking a sufficiently small value $\epsilon>0$ such that the shortest curve $\beta_{0}$ in the homotopy class of $\beta$ in $\mathfrak{B}_{\epsilon}$ is simplified. Suppose that $\beta_{0}$ is simplified and it decomposes as $\beta_{0}=\gamma_{1}\cup\delta_{1}\cup\gamma_{2}\cup\dots\cup\delta_{n}\cup\gamma_{n+1}$. First, suppose the vertical and horizontal lines of the tiling $\mathcal{T}$ are consistently oriented  upwards and leftwards, respectively. Consider the straight line segment $\gamma_{i}$, $1\leq i\leq n+1,$  with the orientation inherited from $\beta_{0}$. Let $p_{i}\in \mathbb{Z}$ be the signed intersection number between $\gamma_{i}$ and the  vertical lines of  $\mathcal{T}$ according to a right hand convention. If $q_{i}\in \mathbb{Z}$ is the corresponding signed intersection number between $\gamma_{i}$ and the horizontal lines of $\mathcal{T}$, we get the pair of integers $(p_{i},q_{i})$ describing the segment  $\gamma_{i}$, and we shall call it the \emph{slope} of $\gamma_{i}$. The name slope suggests that the number $p_{i}/q_{i}\in\mathbb{Q}\cup\{\infty\}$ is close to the slope of the segment  $\gamma_{i}$ in the plane, in fact, it corresponds to the slope of the segment in the plane connecting the points $w_{i-1}$ and $w_{i}$ (the segment $\lambda_{i}$ in Figure \ref{fig:b4}).  \par
For $i\in\{1,2,\dots,n\}$, consider the subcurve of $\beta_{0}$, $\delta_{i}\subset \partial D_{\epsilon}(w_{i})$. There exists a non-negative integer $m_{i}$ such that the angle $r_{i}$ covered by $\delta_{i}$ around $w_{i}$ satisfies $(m_{i}-1)\pi < |r_{i}|\leq m_{i}\pi$. We assign to $\delta_{i}$ the integer $m_{i}$ if it turns counterclockwise around $\partial D_{\epsilon}(w_{i})$  and $-m_{i}$ otherwise. We shall call this integer the \emph{winding number} of $\delta_{i}$ around $w_{i}$. \par
From the decomposition $\beta_{0}=\gamma_{1}\cup\delta_{1}\cup\gamma_{2}\cup\dots\cup\delta_{n}\cup\gamma_{n+1}$ we get the parameterization $(p_{1},q_{1},m_{1}, p_{2},q_{2},\dots, m_{n},p_{n+1},q_{n+1})$, which represent the ordered sequence of slopes and winding numbers of the subcurves of $\beta_{0}$. It is possible to recover the curve $\beta_{0}$ from the ordered sequence of integers and therefore the homotopy class of $\beta$. Moreover, since we required the curve $\beta_{0}$ to be simplified, this representation of the homotopy class of $\beta$ is well-defined and canonical. 
\end{proof}
Note that that $m_{i}=0$ only if $\delta_{i}$ is a point and $\gamma_{i}\cup \gamma_{i+1}$ is a line segment tangent to $\partial D_{\epsilon}(w_{i})$ at $\delta_{i}$. In case that one of $p_{i}$ or $q_{i}$ is equal to zero for some $i$, then the other integer must be $\pm 1$. It occurs $p_{i}=\pm q_{i}$ only if $p_{i},q_{i}\in\{1,-1\}$. In case $p_{i},q_{i}\neq 0 $ and $p_{i}\neq\pm q_{i}$, it follows that they are relatively prime. As an example, the parameterization of the homotopy class of the curve $\beta$ in Figure \ref{fig:b2} is $(2,1,1,1,1,1,0,1,1,-1,1,-1,-1,3,3,-1,-1,-2,-1,1,1,0,1,-3,1)$.

 %Let $\pi_{x}$ and $\pi_{y}$ be the projections of $\mathbb{R}^{2}$ onto the $x$-axis and $y$-axis, respectively. We can assume that $\beta$ is in Morse position with respect to $\pi_{x}$ and $\pi_{y}$, and the critical points appear at different points $X =\{x_{1},\ldots,x_{n}\}\subset [0,1]$, where $x_{i}< x_{j}$ if $i<j$, namely the point $\beta(x_{i})$ is a minimum, a maximum or a saddle point of the functions $\pi_{x}$ or $\pi_{y}$ restricted to $\beta([0,1])$, for every $x_{i} \in X$.\par
%In this section we will prove some results that will be used later.\par

\section{A parameterization of $(1,1)$-knots}\label{sec:3}
%
%A \emph{$(1,1)$-knot} or \emph{$1$-bridge torus knot} in $\mathbb{S}^{3}$ is a knot that admits a $1$-bridge presentation with respect to a standard torus in $\mathbb{S}^{3}$. Equivalently, if $K$ is a $(1,1)$-knot, then it is possible to embed it in a product $T\times [0,1]$, where $T$ is a Heegaard torus of $\mathbb{S}^{3}$ and if $\rho$ is the projection of the manifold $T\times [0,1]$ onto the factor $I=[0,1]$, then the restriction of $\rho$ to the submanifold $K$ has only one maximum $y_{1}$ and one minimum $y_{0}$ with values $1$ and $0$, respectively.  The points $y_{0}$ and $y_{1}$ segment $K$ into two arcs $K_{0}$ and $K_{1}$. For each $r\in [0,1]$, let $T_{r}$ be the level torus $T\times\{r\}\subset T\times [0,1]$, say $T=T_{1/2}$. Each one of the arcs $K_{0}$ and $K_{1}$ intersects transversally every level torus $T_{r}$ in one points for every $r\in (0,1)$. Suppose we isotope the knot $K$, preserving the bridge position, such that $K_{0}$ is a straight arc in $T\times [0,1]$, that is to say, $K_{0}= \{w_{0}\}\times [0,1]$ for some fixed point $w_{0}\in T$. In this case, we say that $K$ is in \emph{simple bridge position} with respect to $T$. If $\pi:T\times I\rightarrow T$ is the projection onto the torus, then $\pi(K_{1})\subset T$ is a smooth curve intersecting $\{w_{0}\}$ only at its endpoints. \par

A \emph{$(g,1)$-knot} in $\mathbb{S}^{3}$ is a knot that admits a $1$-bridge presentation with respect to a genus-$g$ Heegaard surface  $\Sigma^{g}\subset\mathbb{S}^{3}$. Equivalently, if $K$ is a $(g,1)$-knot, then it is possible to embed it in a product $\Sigma^{g}\times [0,1]$, and if $\rho$ is the projection of the manifold $\Sigma^{g}\times [0,1]$ onto the factor $I=[0,1]$, then the restriction of $\rho$ to the submanifold $K$ has only one maximum $y_{1}$ and one minimum $y_{0}$ with values $1$ and $0$, respectively.  The points $y_{0}$ and $y_{1}$ segment $K$ into two arcs $A_{0}$ and $A_{1}$. For each $t\in [0,1]$, let $\Sigma^{g}_{t}$ be the level surface $\Sigma^{g}\times\{t\}\subset \Sigma^{g}\times [0,1]$, say $\Sigma^{g}=\Sigma^{g}_{1/2}$. Each one of the arcs $A_{0}$ and $A_{1}$ intersects transversely the level surface $\Sigma^{g}_{t}$ in one point for every $t\in (0,1)$. Suppose we isotope the knot $K$, preserving the bridge position, such that $A_{0}$ is a straight arc in $\Sigma^{g}\times I$, that is to say, $A_{0}= \{w_{0}\}\times I$ for some fixed point $w_{0}\in \Sigma^{g}$. In this case, we say that $K$ is in \emph{straight bridge position} with respect to $\Sigma^{g}$. If $\pi:\Sigma^{g}\times I\rightarrow \Sigma^{g}$ is the projection onto the surface, then $\pi(A_{0})=\{w_{0}\}$, while  $\pi(A_{1})\subset \Sigma^{g}$ is a  curve intersecting $\{w_{0}\}$ only at its endpoints. \par

Now we restrict to the case $g=1$, and to simplify notation we shall use $T$ to refer the genus-$1$ Heegaard surface $\Sigma^{1}$. Suppose $K=A_{0}\cup A_{1}$ is a $(1,1)$-knot which is in straight bridge position with respect to $T$. Let $\mu$ and $\lambda$ form a meridian-longitude curve system for $T$, that is to say, both are simple closed curves intersecting each other in one point, and each one of this curves bounds a meridian disk in one of the two solid torus in the complement of $T$ in the 3-sphere. % Moreover, there is a parameterization $f$ of the torus such that $T=\{f(e^{i\theta}, e^{i\eta})\mid \theta, \eta\in [0,2\pi]\}$,  $\mu=\{f(e^{i\theta},1)\mid \theta\in [0,2\pi]\}$ and $\lambda=\{f(1,e^{i\eta})\mid \eta\in [0,2\pi]\}$.
 The space we get  after cutting the torus $T$ along the curves $\mu$ and $\lambda$ can be modeled by a unitary square $Q$ with paired opposite edges which correspond to the curves $\lambda$ and $\mu$. We can assume that $w_{0}$ is the central point of $Q$. Let $\tilde{T}$ be the universal covering space of $T$, tiled with copies of $Q$ and corresponding covering map $\varphi:\tilde{T}\rightarrow T$. Namely, we represent $\tilde{T}$ by a plane with the unitary square tiling $\mathcal{T}$ as in Section \ref{sec:2}.   \par

\begin{teo}\label{teo:31} Let $K$ be a  $(1,1)$-knot. Then $K$ is parameterized by an ordered sequence of integers $(p_{1}, q_{1}, m_{1}, p_{2}, q_{2},\ldots,m_{n},p_{n+1},q_{n+1})$, for some $n\geq 0$, such that $p_{i}$ and $q_{i}$ are not both zero, $i=1,\dots,n+1$.
\end{teo}

\begin{proof}

Consider the following projections onto the torus $T$ that were described before. \par

\[
\begin{tikzcd}
    T\times I \arrow[swap]{dr}{\pi} & \tilde{T} \arrow{d}{\varphi} \\
     & T
\end{tikzcd}
\]
Suppose that $K=A_{0}\cup A_{1}$ is a straight bridge position with respect to $T$ and  the arc $A_{1}$ is parameterized by a smooth function $\alpha:I\rightarrow T\times I$ such that $\alpha(t)\in T\times\{t\}$ for every $t\in I$. Let $\tilde{w}_{0}\in \varphi^{-1}(w_{0})$ be a fixed point and let $\beta: I\rightarrow\tilde{T}$ be the lift of the curve $\pi\circ\alpha$ starting at $\tilde{w}_{0}$ and ending at some point $\tilde{w}_{1}\in\varphi^{-1}(w_{0}) $ such that $\pi\circ \alpha(t)=\varphi\circ\beta(t)$ for every $t$. From Proposition \ref{teo:p1} it follows that there exists $\epsilon>0$ such that in the $\epsilon$-punctured plane $\mathfrak{B}_{\epsilon}\subset\tilde{T}$, where the punctures are centered at the points in $\varphi^{-1}(w_{0})$, the shortest curve $\beta_{0}$ in the homotopy class of the restriction of $\beta$ to $\mathfrak{B}_{\epsilon}$ is simplified. Let $\bar{\beta}$ be the restriction of $\beta$ to $\mathfrak{B}_{\epsilon}$, which is defined in an interval $\bar{I}=[\mu_{1},1-\mu_{2}]$ for some $\mu_{1},\mu_{2}>0$, sufficiently small. Let $\tilde{H}: \bar{I}\times I\rightarrow \mathfrak{B}_{\epsilon}$ be a homotopy between the curves  $\bar{\beta}$ and $\beta_{0}$ in $\mathfrak{B}_{\epsilon}$, such that $\tilde{H}(t,0)=\bar{\beta}(t)$ and $\tilde{H}(t,1)=\beta_{0}(t)$. Define the continuous functions $\alpha_{0}:\bar{I}\rightarrow T\times I$ by $\alpha_{0}(t)=(\varphi\circ \beta_{0}(t),t)$, and   $H:\bar{I}\times I\rightarrow T\times I$ by $H(t,s)= (\varphi\circ \tilde{H}(t,s),t)$. The function $H$ is a homotopy between $\bar{\alpha}$, the restriction of $\alpha$ into $(T\setminus D_{\epsilon}(w_{0}))\times I$, and the curve $\alpha_{0}$, which projects under $\pi$ onto the projection of $\beta_{0}$ under $\varphi$. Moreover, since the image of the curve $\bar{\alpha}$ along the homotopy is always transversal to the level tori in $T\times I$, $H$ is in fact an isotopy between $\bar{\alpha}$ and $\alpha_{0}$.\par
A parameterization for the knot $K$ is given by the parameterization $(p_{1}, q_{1}, \\m_{1}, p_{2}, q_{2},\ldots,m_{n},p_{n+1},q_{n+1})$ of the homotopy class of the curve $\beta_{0}$ in $\mathfrak{B}_{\epsilon}$ as shown in   Proposition \ref{teo:p1}.
\end{proof}

If $K=A_{0}\cup A_{1}$ is a $(1,1)$-knot in straight bridge position with respect to $T$ and $A_{1}$ is represented by a smooth function $\alpha_{0}$ realizing the parameterization of Theorem \ref{teo:31}, we shall say that the presentation (or position) of the knot is \emph{tight} and the parameterization induced will be called a \emph{tight parameterization} of $K$, in reference to the minimal-length property of the associated curve in the multipunctured plane that induces the parameterization. \par
Note that not every sequence of integers $(p_{1}, q_{1}, m_{1}, p_{2}, q_{2},\ldots,m_{n},p_{n+1},q_{n+1})$ describes a minimal-length curve in the multipunctured plane, but if it does then there is only one $(1,1)$-knot associated to this sequence of integers according to the relation between $(1,1)$-knots and curves in the multipunctured plane described in the proof of the previous theorem. However, it is not clear when two different sequences of integers produce the same knot. In the following section we will see that in the case of satellite $(1,1)$-knots the proposed parameterization is essentially unique for each knot in the family.\par
The parameterization of $(1,1)$-knots that we have introduced has an unbounded number of parameters in contrast with other  parameterizations of 1-bridge torus knots (see for instance Sections 3 and 4 from \cite{CK}, where a parameterization for $(1,1)$-knots requires four integers and a sign). It is not clear how this  new parameterization relates to classical presentations of 1-bridge torus knots, such as the Schubert's or Conway's normal forms (see \cite{CK}), or the mapping class group of the twice punctured torus (as described in \cite{CM}).

\section{The case of satellite $(1,1)$-knots}\label{sec:4}

K. Morimoto and M. Sakuma \cite{MS} introduced the useful description of satellite $(1,1)$-knots as satellites of torus knots with rational-link patterns. Let $K_{0}$ be a non-trivial $(p,q)$-torus knot in $\mathbb{S}^3$, and let $K_{1}\cup K_{2}$ be a rational link of type $(\alpha,\beta)$, $\alpha \geq 4$, in $\mathbb{S}^{3}$. Consider the orientation preserving homeomorphism $\varphi: E(K_{1})\rightarrow N(K_{0})$ which takes a meridian $m\subset \partial E(K_{1})$ of $K_{1}$ to a fiber $l\subset\partial N(K_{0})=\partial E(K_{0})$ of the Seifert fibration $D(-r/p,s/q)$ of $E(K_{0})$. The knot $\varphi(K_{2})\subset N(K_{0})\subset \mathbb{S}^{3}$ is a satellite $(1,1)$-knot that will be denoted by $K(\alpha,\beta;p,q)$, and every satellite $(1,1)$-knot admits one of these representations.\par

Before we proceed with the parameterization of satellite $(1,1)$-knots, we present a brief reminder of continued fractions. Given a finite sequence of non-zero integers $\{a_{i}\}_{i=1}^{n}$, we produce the continued fraction  $[a_{1},a_{2},\ldots,a_{n}]:=a_{1}+(a_{2}+(\cdots + a_{n}^{-1})^{-1}\cdots)^{-1}$, which can be simplified to a rational number $p_{n}/q_{n}$, where $p_{n}$ and $q_{n}$ are relatively prime. Truncating the sequence at $a_{k}$, $k\leq n$, and simplifying the truncated continued fraction produces the \textit{$k$-th convergent} $p_{k}/q_{k}=[a_{1},\ldots,a_{k}]$. The general expressions of the fist convergents are $a_{1}/1$, $(a_{1}a_{2}+1)/a_{2}$, $[a_{1}(a_{2}a_{3}+1)+a_{3}]/(a_{2}a_{3}+1)$, and so forth; it is easy to prove by induction that these expressions are simplified, namely, the numerator and denominator in each expression are relatively prime. The following Lemma concerning convergents of continued fractions is well known and can be proven by induction:
\begin{lem}\label{lem:41} If $p_{k}/q_{k}$ corresponds to the $k$-th convergent of $[a_{1},a_{2},\ldots,a_{n}]$, then:
\begin{itemize}
\item[(i)] $p_{k+1}=a_{k+1}p_{k}+p_{k-1}$ and $q_{k+1}=a_{k+1}q_{k}+q_{k-1}$ for $2\leq k\leq n-1$.
\item[(ii)] $p_{k}q_{k+1}-p_{k+1}q_{k}=(-1)^{k}$ for $k\leq n-1$. 
\end{itemize}
\end{lem}

%An observation from (ii) in the previous Lemma is that for any $k$, in the $k$-th convergent $p_{k}$ and $q_{k}$ are relatively prime. 
The following theorem is known as the Palindrome Theorem (see \cite{KL}, Theorem 4) and will be useful in our further analysis. For completeness we present an elementary proof.

\begin{pr}[Palindrome Theorem]\label{pr:pt} If $p_{k}/q_{k}$ is the $k$-th convegent of $[a_{1},a_{2},\ldots,a_{n}]$, $k\leq n$, then the reversed continued fraction $[a_{k},a_{k-1},\ldots,a_{2},a_{1}]$ equals $p_{k}/p_{k-1}$ and $q_{k}p_{k-1}\equiv (-1)^{k-1}\mod |p_{k}|$.
\end{pr}

\begin{proof} We proceed by induction on $k$. For $k=2$, the result is obvious. Suppose that $p_{k}/p_{k-1}=[a_{k},a_{k-1},\ldots,a_{1}]$ and $q_{k}p_{k-1}\equiv (-1)^{k-1}\mod |p_{k}|$, for some $k\geq 2$. Then

$$ [a_{k+1},\ldots,a_{1}] = a_{k+1}+ \frac{p_{k-1}}{p_{k}}= \frac{a_{k+1}p_{k}+p_{k-1}}{p_{k}}=\frac{p_{k+1}}{p_{k}}$$
where the last equality follows from Lemma \ref{lem:41}(i). Finally, from the part (ii) in the same lemma it follows that $q_{k+1}p_{k}-p_{k+1}q_{k}=(-1)^{k}$, and then $q_{k+1}p_{k}\equiv (-1)^{k}\mod |p_{k+1}|$. 

\end{proof}

%A first attempt to identify two equivalent parameterizations for a given satellite $(1,1)$-knot is through the use of Theorem \ref{teo:41}, the result of Morimoto-Sakuma and the classification of rational knots and links by Schubert \cite{Sch}. However, in the general case of $(1,1)$-knots, the problem gets complex.\par

Let $K=K(\alpha,\beta;p,q)$ be a satellite $(1,1)$-knot with rational-link pattern $K_{1}\cup K_{2}$, such that $K=\varphi(K_{2})\subset\mathbb{S}^{3}$, where $\varphi:E(K_{1})\rightarrow N(K_{0})$ is the homeomorphism between the exterior of $K_{1}$ and a regular neighborhood of the $(p,q)$-torus knot $K_{0}$ as before. Once the companion $(p,q)$-torus knot $K_{0}$ is fixed, the description of $K$ relies on the $(\alpha,\beta)$-rational link $K_{1}\cup K_{2}$, as in the analysis of the patterns developed in \cite{FE}. According to Lemma 2.1 from \cite{FE}, the rational link $K_{1}\cup K_{2}$ admits a diagram that is described by a sequence of an odd number of non-zero even integers $A=(c_{1}, d_{1},c_{2},d_{2}, \ldots, c_{n},d_{n},c_{n+1})$, corresponding to a sequence of descending crossings of the model in Figure 1 from \cite{FE}. 
%such that $\alpha/\beta =[c_{1}, d_{1},c_{2}, \ldots, d_{n},c_{n+1}]$. 
Note that $c_{i}$ corresponds to crossings between the two components of $K_{1}\cup K_{2}$, while $d_{i}$ represents crossings of one of the components with itself, say $K_{2}$, for every possible $i$. It is clear that the pattern defined by the sequence $A$ corresponds to a straight bridge position for the knot $K$; furthermore, we shall see that it determines a tight presentation for $K$ (Algorithm \ref{al:1}). \par 

On the other hand, suppose $K=A_{0}\cup A_{1}$ is a tight presentation of $K$, where $A_{0}$ is an arc of the form $\{w_{0}\}\times I\subset T\times I$ for some $w_{0}$ in the standard torus $T$ as in Section \ref{sec:3}. The pattern associated to this presentation, $K_{1}'\cup K_{2}'$, can be represented by a diagram where $K_{2}'$ splits as $K_{2}'=B_{0}\cup B_{1}$, with $B_{0}$ a vertical arc corresponding to $A_{0}$ and without crossings with $K_{1}'$, while $B_{1}$ is a monotonous arc  which is winding around the component $K_{1}'$ and the arc $B_{0}$. This diagram of $K_{1}'\cup K_{2}'$ can be represented by a sequence of non-zero even integers as before. We will establish relations between all the presentations of patterns for $K$ obtained in this manner.

 % It is also true that given a sequence of non-zero even integers $(c_{1}, d_{1},c_{2},d_{2}, \ldots, c_{n},d_{n},c_{n+1})$ satisfying $\alpha/\beta =[c_{1}, d_{1},c_{2}, \ldots, d_{n},c_{n+1}]$, then there exist a unique tight presentation of $K$ corresponding to this sequence. 
%  then the knot $K_{2}$ splits as $K_{2}=\varphi^{-1}(A_{0})\cup\varphi^{-1}(A_{1})$, and there exist a diagram of $K_{1}\cup K_{2}$ where $\varphi^{-1}(A_{0})$ is represented by a vertical arc without crossings with $K_{1}$, while $\varphi^{-1}(A_{1})$ is a monotonous arc  which is winding around the component $K_{1}$ and the arc $\varphi^{-1}(A_{0})$. This diagram of $K_{1}\cup K_{2}$ can be represented by a sequence of non-zero integers as before. % It is also true that given a sequence of non-zero even integers $(c_{1}, d_{1},c_{2},d_{2}, \ldots, c_{n},d_{n},c_{n+1})$ satisfying $\alpha/\beta =[c_{1}, d_{1},c_{2}, \ldots, d_{n},c_{n+1}]$, then there exist a unique tight presentation of $K$ corresponding to this sequence.

\begin{ob}\label{ob:1} Let $\alpha/\beta=[c_{1}, d_{1},c_{2}, \ldots, d_{n},c_{n+1}]$, where $c_{i}$ and $d_{j}$ are non-zero even numbers, for every possible $i$ and $j$.
\begin{itemize}
\item[(i)] $|\alpha|\geq |\beta|$, since $|c_{1}|\geq 2$.
\item[(ii)]The sequence of non-zero even numbers $(c_{1}, d_{1},c_{2}, \ldots, d_{n},c_{n+1})$ is unique for $(\alpha,\beta)$. This follows from the proof of Lemma 2.1 in \cite{FE}, which is based on the Euclidean division algorithm.
\end{itemize}
\end{ob}

Consider a sequence of $2n+1$ even integers $A=(c_{1}, d_{1},c_{2}, \ldots, d_{n},c_{n+1})$, where $c_{i}\neq 0$ for all $i$, but it could happen $d_{j}=0$ for one or more values $j$. We describe three elementary operations on the sequence $A$: 
\begin{itemize}
\item[(i)] Suppose $|c_{i}|>2$ for some $i\in\{1,2\ldots, n+1\}$. An \textit{expansion at $c_{i}$} of $A$ will be a substitution of $c_{i}$ in the original sequence $A=(A_{1},c_{i},A_{2})$ for an alternating subsequence of length $|c_{i}|-1$ of the form $(2,0,2,\ldots,0,2)$ if $c_{i}>0$ or $(-2,0,-2,\ldots,0,-2)$ if $c_{i}<0$, to obtain a new sequence $A'=(A_{1},\pm2,0,\pm2,\ldots,0,\pm2 ,A_{2})$. If there is another value $|c_{j}|>2$ in $A$ we can proceed with another expansion on $A'$, and so forth until we obtain a sequence $A_{e}$, where no more expansions are possible, namely, if $A_{e}=(g_{1}, h_{1},g_{2}, \ldots, h_{l},g_{l+1})$, then $g_{i}\in\{2,-2\}$ for all $i$. We shall call $A_{e}$ the \emph{expanded form} of $A$.     
\item[(ii)] Suppose $A$ contains a maximal alternating subsequence of length $2k-1$, $k\geq 2$, of the form $(2,0,2,\ldots,0,2)$ or $(-2,0,-2,\ldots,0,-2)$, then the sequence $A'$ obtained after substituting this sequence for the length-$1$ subsequence $(2k)$ or $(-2k)$ in $A$, respectively, will be called a \textit{contraction} of $A$. If we continue realizing contractions until we get a sequence $A_{c}$, where no more contractions are possible, then $A_{c}$ will be called the \emph{contracted form} of $A$. Note that in a sequence $A=(c_{1}, d_{1},c_{2}, \ldots, d_{n},c_{n+1})$, where $c_{i},d_{j}\neq 0$ for every $i$ and $j$, then $(A_{e})_{c}=A$.       
\item[(iii)] Consider a sequence of $2n+1$ non-zero even numbers $A=(c_{1}, d_{1},c_{2}, \ldots, d_{n},c_{n+1})$. Let $A_{e}=(g_{1}, h_{1},g_{2}, \ldots, h_{l},g_{l+1})$ be the expanded form of $A$. Define the transformation $f$ on $A_{e}$ to obtain the sequence $f(A_{e})$ as follows: change in $A_{e}$ each value $g_{i}$ for $-g_{i}$, and change the value $h_{i}$ for $h_{i}+(g_{i}+g_{i+1})/2$. Note that $f(A_{e})$ is a sequence in expanded form of length $2l+1$. The contracted form of $f(A_{e})$, denoted by $(f(A_{e}))_{c}$, will be called the \textit{sequence associated} to $A$. 
\end{itemize} 

%We show how the operations on sequences described before are related to our study of satellite $(1,1)$-knots. Let $K_{1}\cup K_{2}$ be the rational-link pattern corresponding to the satellite , and suppose that the this rational  link is represented by a diagram with a sequence of non-zero even integers $(c_{1}, d_{1},c_{2},d_{2}, \ldots, c_{n},d_{n},c_{n+1})$. We can place this diagram such that it has only two maxima and two minima. Moreover, if $x_{0}$ and $x_{1}$ are the minimum and maximum, respectively, corresponding to the component $K_{2}$, then these points divide $K_{2}$ into two arcs $B_{0}$ and $B_{1}$, where we place $B_{0}$ to a vertical arc, while $B_{1}$ is an arc which is winding around the component $K_{1}$ and the arc $B_{0}$. Under these assumptions, we have the following result:
 
We show how the operations on sequences described before are related to the study of rational links. Let $K_{1}\cup K_{2}$ be the two-components rational link  represented by a diagram with a sequence of non-zero even integers $(c_{1}, d_{1},c_{2}, \ldots, c_{n},d_{n},c_{n+1})$ as before. Suppose that $K_{2}=B_{0}\cup B_{1}$, where $B_{0}$ is a vertical arc, while $B_{1}$ is a monotonous arc which is winding around the component $K_{1}$ and the arc $B_{0}$. Under these assumptions, we have the following result:

%We can place this diagram such that it has only two maxima and two minima. Moreover, if $x_{0}$ and $x_{1}$ are the minimum and maximum, respectively, corresponding to the component $K_{2}$, then these points divide $K_{2}$ into two arcs $B_{0}$ and $B_{1}$, where we place $B_{0}$ to a vertical arc, while $B_{1}$ is an arc which is winding around the component $K_{1}$ and the arc $B_{0}$. Under these assumptions, we have the following result:

\begin{lem}\label{lem:as} Let $K_{1}\cup K_{2}$ be a rational two-components link represented by a diagram with a sequence of non-zero even integers $A=(c_{1}, d_{1},c_{2},d_{2}, \ldots, c_{n},d_{n},c_{n+1})$. The link represented by the sequence of integers associated to $A$, $A'=(f(A_{e}))_{c}$, is isotopic to $K_{1}\cup K_{2}$. 
\end{lem}

\begin{proof}
Suppose that the extended form of $A$ is $A_{e}=(g_{1}, h_{1},g_{2}, \ldots, h_{l},g_{l+1})$, then the extended form of $A'$ has the same length as $A_{e}$; moreover, $A'_{e}=(-g_{1}, h_{1}',-g_{2}, \ldots, h_{l}',-g_{l+1})$, and $h_{i}'=h_{i}+(g_{i}+g_{i+1})/2$. The isotopy we require is an isotopy that swaps the roles of $B_{0}$ and $B_{1}$ in $K_{2}$, namely, after the isotopy we can represent $B_{1}$ by a vertical arc and $B_{0}$ by an arc winding around $K_{1}$ and $B_{1}$. This isotopy may be accomplished unwrapping the arc $B_{1}$ in descending (or ascending) direction and the numbers in the sequence $A'$ are obtained.       
\end{proof}

\begin{al}\label{al:1}
Let $K=K(\alpha,\beta;p,q)$ be a satellite $(1,1)$-knot in a tight presentation $K=A_{0}\cup A_{1}$. Suppose that the associated rational-link pattern is $K_{1}\cup K_{2}$ is described by the sequence of non-zero even integers $A=(c_{1}, d_{1},c_{2}, \ldots, c_{n},d_{n},c_{n+1})$.  We describe an algorithm to obtain the tight parameterization of $K=A_{0}\cup A_{1}$ associated to the sequence $A$:

\begin{itemize}
\item[(1)] Obtain the extended form $A_{e}=(g_{1}, h_{1},g_{2}, \ldots, h_{l},g_{l+1})$ of $A$. Remember that all the elements in the sequence are even integers and $g_{i}\in\{2,-2\}$, for every $i$.
\item[(2)] The tight parameterization of $K$ will have length $3l+2$, where the parameters in positions $3k+1$ and $3k+2$ are $(g_{k+1}p)/2$ and $(g_{k+1}q)/2$, respectively, for $k=0,1,\ldots,l$.
\item[(3)]  If $sgn$ is the usual sign function defined by $sgn(x)=-1,0,1$ if $x<0,x=0$ or $x>0$, respectively, then the third parameter will be $sgn(g_{1})$ if $h_{1} = 0$, $h_{1}+sgn(h_{1})$ if $sgn(g_{1})=sgn(g_{2})=sgn(h_{1})$,  $h_{1}-sgn(h_{1})$ if $sgn(g_{1})=sgn(g_{2})\neq sgn(h_{1})$, and $h_{1}$ if $sgn(g_{1})\neq sgn(g_{2})$.  Analogously, the $3l$-th parameter will be $sgn(g_{1+1})$ if $h_{l} = 0$, $h_{l}+sgn(h_{l})$ if $sgn(g_{l+1})=sgn(g_{l})=sgn(h_{l})$,  $h_{l}-sgn(h_{l})$ if $sgn(g_{l+1})=sgn(g_{l})\neq sgn(h_{l})$, and $h_{l}$ if $sgn(g_{l+1})\neq sgn(g_{l})$.

\item[(4)] The $3k$-th parameter for $k=2,3,\ldots,l-1$, will be:
\begin{itemize}
\item[-] if $sgn(g_{k})=sgn(g_{k+1})$:
\begin{itemize}
\item[-] if $h_{k}=0$:
\begin{itemize}
\item[-] if $h_{k-1}\neq 0$ and $sgn(h_{k-1})\neq sgn(g_{k})$, or if $h_{k+1}\neq 0$ and $sgn(h_{k+1})\neq sgn(g_{k})$: $sgn(g_{k})$
\item[-] in other case: $0$
\end{itemize} 
\item[-] if $h_{k} \neq 0$ and $sgn(h_{k})\neq sgn(g_{k})$:
\begin{itemize}
\item[-] if  $sgn(h_{k-1})\neq sgn(h_{k})$ or  $sgn(h_{k+1})\neq sgn(h_{k})$: $g_{k}-sgn(g_{k})$
\item[-] in other case: $g_{k}-2sgn(g_{k})$
\end{itemize} 
\item[-] if $h_{k} \neq 0$ and $sgn(h_{k})=sgn(g_{k})$:
\begin{itemize}
\item[-] if $h_{k-1}\neq 0$ and $sgn(h_{k-1})\neq sgn(h_{k})$, or if $h_{k+1}\neq 0$ and $sgn(h_{k+1})\neq sgn(g_{k})$: $g_{k}+sgn(g_{k})$
\item[-] in other case: $g_{k}$
\end{itemize}
\end{itemize} 

\item[-] if $sgn(g_{k})\neq sgn(g_{k+1})$:
\begin{itemize}
\item[-] if $sgn(h_{k})=sgn(g_{k})$:
\begin{itemize}
\item[-] if $h_{k-1}\neq 0$ and $sgn(h_{k-1})\neq sgn(h_{k})$, or if $sgn(h_{k+1})\neq sgn(h_{k})$: $g_{k}$
\item[-] in other case: $g_{k}-sgn(g_{k})$
\end{itemize}
\item[-] if $sgn(h_{k})=sgn(g_{k+1})$:
\begin{itemize}
\item[-] if $h_{k+1}\neq 0$ and $sgn(h_{k+1})\neq sgn(h_{k})$, or if $sgn(h_{k-1})\neq sgn(h_{k})$: $g_{k}$
\item[-] in other case: $g_{k}-sgn(g_{k})$
\end{itemize}

\end{itemize}
\end{itemize}
% first parameter will be $sgn(g_{1})$ if $h_{1} = 0$, $h_{1}+sgn(h_{1})$ if $sgn(g_{1})=sgn(g_{2})=sgn(h_{1})$,  $sgn(h_{1})$ if $sgn(g_{1})=sgn(g_{2})\neq sgn(h_{1})$  where $sgn$ is the usual sign function. Analogously, the $3l$-th parameter will be $sgn(h_{l})$ if  $h_{l}\neq 0$, and $sgn(g_{l+1})$ if $h_{l}=0$
\end{itemize}
\end{al}

The algorithm presented above may seem cumbersome but it is obtained from a straightforward process relating the tight presentation of $K$ with $A$. An advantage of this algorithm is  that it is ready to be implemented in a computer program. We present the main theorem in this section which shows that in the case of satellite $(1,1)$-knots, a tight parameterization  is essentially unique.

\begin{teo}\label{teo:up} Let $K=K(\alpha,\beta;p,q)$ be a satellite $(1,1)$-knot. There exist two tight parameterizations of $K$. If $K=A_{0}\cup A_{1}$ is a tight presentation of $K$, the two tight parameterizations are related by an isotopy that swaps the roles of $A_{0}$ and $A_{1}$.   
\end{teo}

\begin{proof}
In Theorem 1.2 from \cite{Sa}, it was demonstrated that the only companion knot of $K$ is the non-trivial $(p,q)$-torus knot. This implies that if $K$ has two rational-link patterns $K_{1}\cup K_{2}$ and $K_{1}'\cup K_{2}'$, then these two links must be isotopic. According to Lemma 2.1 in \cite{FE}, we can choose a rational-link pattern for $K$, $K_{1}\cup K_{2}$, which is represented by sequence of non-zero even integers $A=(c_{1}, d_{1},c_{2}, \ldots, c_{n},d_{n},c_{n+1})$. We can assume $\alpha/\beta=[c_{1}, d_{1},c_{2}, \ldots, c_{n},d_{n},c_{n+1}]$, where $\alpha>0$, then $\alpha>|\beta|$ as seen in Remark \ref{ob:1}. Remember that the sequence $A$ defines a tight presentation of the knot $K=A_{0}\cup A_{1}$ (Algorithm \ref{al:1}).\par 
Suppose that $K_{1}'\cup K_{2}'$ is another pattern for $K$ defined by a sequence of non-zero even integers $A'=(e_{1}, f_{1},e_{2}, \ldots, e_{m},f_{m},e_{m+1})$ such that $\alpha'/\beta'=[e_{1}, f_{1},e_{2}, \ldots, e_{m},f_{m},e_{m+1}]$ with $\alpha'>0$ and $\alpha'>|\beta'|$. According to the Schubert's classification of rational knots and links, it must be $\alpha=\alpha'$, and either  $\beta\equiv \beta' \mod \alpha$  or $\beta\beta'\equiv 1\mod \alpha$ (see \cite{Sch} or Theorem 2 in \cite{KL}). There are at most three possibilities for $\beta'$: $\beta$, $\beta^{-1}$ and $\beta\pm\alpha$ (plus sign if $\beta<0$ and minus in other case). Since given the numerator and denominator the sequence of non-zero even numbers is unique according to Remark \ref{ob:1}, then each value of $\beta'$ define a unique diagram of $K_{1}\cup K_{2}$.\par
Let us first consider the case $\beta'=\beta^{-1}$. From the Palindrome Theorem (Proposition \ref{pr:pt}), it follows that the sequence $A'$ must coincide with the reversed  sequence of $A$. In the diagrams of the patterns associated to $A$ and $A'$, this operation corresponds to an isotopy that rotates $180^{\circ}$ the plane that contains the diagrams, followed by a rotation of $180^{\circ}$ around a vertical axis. This isotopy between the patterns can not be realized as an isotopy of the knot $K$, unless $A=A'$.  \par
Finally, suppose that $\beta'=\beta\pm\alpha$. From Lemma \ref{lem:as} we know that $(f(A_{e})_{c})$ represent a link isotopic to $K_{1}\cup K_{2}$. Since the sequence $(f(A_{e})_{c})$ is distinct from $A$, it must be $A'=(f(A_{e})_{c})$. Remember that $(f(A_{e})_{c})$ is obtained from an isotopy in $K_{1}\cup K_{2}$, where $K_{2}=B_{0}\cup B_{1}$ and $B_{0}$ is a vertical arc, that swaps the roles of $B_{0}$ and $B_{1}$, namely, it isotopes $B_{1}$ into a vertical arc. This isotopy can be realized as an isotopy that swaps the roles of $A_{0}$ and $A_{1}$ in $K$. 

%Furthermore, we can describe the action of the isotopy that switchs the roles of $A_{0}$ and $A_{1}$ on the two tight parameterizations of the knot $K$. If $(e_{1}p,e_{1}q,m_{1},e_{2}p,e_{2}q,m_{2},\ldots,m_{n},e_{n+1}p,e_{n+1}q)$, where $e_{i}\in\{1,-1\}$, is a tight parameterization obtained from the sequence $A$, then the tight parameterization obtained from $(f(A_{e}))_{c}$ will be the same except for a change of $e_{i}$ for $-e_{i}$, for $i=1,2,\ldots,n+1$. This result would be expected and follows form the algorithm 

%there exists only one more tight parameterization of $K$ obtained after changing $d_{i}$ for $-d_{i}$, $i=1,2\ldots, n+1$ in the original parameterization, which corresponds to change the roles of $A_{0}$ and $A_{1}$.
\end{proof}

Furthermore, we can describe the action of the isotopy that switches the roles of $A_{0}$ and $A_{1}$ on the two tight parameterizations of the knot $K$ from Theorem \ref{teo:up}. If $(e_{1}p,e_{1}q,m_{1},e_{2}p,e_{2}q,m_{2},\ldots,m_{n},e_{n+1}p,e_{n+1}q)$, where $e_{i}\in\{1,-1\}$, is a tight parameterization obtained from the sequence $A$, then the tight parameterization obtained from $(f(A_{e}))_{c}$ will be the same except for a change of $e_{i}$ for $-e_{i}$,  $i=1,2,\ldots,n+1$. This is what would be the expected and follows directly form the Algorithm \ref{al:1}.\par

In order to illustrate how the processes and algorithms described in this section works, let us see an example. Consider the satellite $(1,1)$-knot $K=(\alpha,\beta,p,q)$ with rational-link pattern defined by the sequence
$$A=(-8,-4,2,4,4,-2,4)$$ 
The succession of steps to obtain the sequence associated to A: 
\begin{align*}
  A_{e}=(-2,0,-2,0,-2,0,-2,-4,2,4,2,0,2,-2,2,0,2) \\
 f(A_{e})=(2,-2,2,-2,2,-2,2,-4,-2,6,-2,2,-2,0,-2,2,-2)\\
 A'= (f(A_{e}))_{c}=(2,-2,2,-2,2,-2,2,-4,-2,6,-2,2,-4,2,-2)
\end{align*} 
The rational numbers obtained from the continued fractions corresponding to $A$ and $A'$  are, respectively, $-6766/817$ and $6766/5949$, as expected from Theorem \ref{teo:up}. Finally, the tight parameterizations deduced from $A$ and $A'$ as in Algorithm \ref{al:1} are, respectively:
\begin{align*}
  (-p,-q,-1,-p,-q,0,-p,-q,0,-p,-q,-4,p,q,5,p,q,1,p,q,-1,p,q,1,p,q) \\
(p,q,-1,p,q,0,p,q,0,p,q,-4,-p,-q,5,-p,-q,1,-p,-q,-1,-p,-q,1,-p,-q)
\end{align*}

Note that both parameterizations differ by a change of sign in each parameter $p$ and $q$, which is consistent with Theorem \ref{teo:up}. Based on Theorem \ref{teo:up}, we propose the following conjecture:

\begin{cnj}\label{teo:cu} Let $K$ be a $(1,1)$-knot in a tight presentation $K=A_{0}\cup A_{1}$. There exist exactly two tight parameterizations of $K$. The two parameterizations are related by an isotopy that swaps the roles of $A_{0}$ and $A_{1}$.   
\end{cnj}

%a  satellite $(1,1)$-knot $K$ can be parameterized by a sequence of integers $(c_{1}p, c_{1}q, m_{1},c_{2} p, c_{2}q,\ldots,m_{n},c_{n+1}p,c_{n+1}q)$, where $(p,q)$ are the parameters of the companion torus knot of $K$. The sequence of integers $(c_{1},m_{1},c_{2},\dots,m_{n},c_{n+1})$ describe completely the rational-link pattern of the satellite knot and the continued fraction we get from the ordered sequence characterizes this rational link according to the characterization of rational knots and links by Schubert \cite{Sch} 

\section{The hyperbolic geoboard model}\label{sec:5}

In the present section we show a route to generalize the representation of $(1,1)$-knots from previous section to the general case of $(g,1)$-knots, $g\geq 1$. Suppose that $K=A_{0}\cup A_{1}\subset\Sigma^{g}\times [0,1]$ is a $(g,1)$-knot in straight bridge position with respect to the genus-$g$ Heegaard surface $\Sigma^{g}$, $g>1$, such that $A_{0}=\{w_{0}\}\times [0,1]$ for some $w_{0}\in \Sigma^{g}$ and $A_{1}$ is an arc transversal to the level surfaces $\Sigma^{g}_{t}=\Sigma^{g}\times \{t\}$, $t\in (0,1)$. We fix a standard set of  $2g$ simple closed geodesics $\Gamma=\{\alpha_{1},\alpha_{2},\ldots,\alpha_{2g}\}$ in $\Sigma^{g}$ having a point in common such that $\Sigma^{g}\setminus \bigcup \alpha_{i}$ is homeomorphic to a disk. \par
 
Consider the hyperbolic unit disk $\mathbb{D}^{2}$ as the universal covering of $\Sigma^{g}$ through the covering map $\varphi$. Suppose $\mathbb{D}^{2}$ is  tessellated by hyperbolic regular $4g$-gons of constant area corresponding to the fundamental domains of $\Sigma^{g}$ obtained after cutting $\Sigma^{g}$ along the curves in $\Gamma$ in the standard fashion (see Figure \ref{fig:h3}). If $\tilde{w}_{0}$ is the center of $\mathbb{D}^{2}$, we can assume that $\varphi(\tilde{w}_{0})=w_{0}$.\par
As we proceed in Section \ref{sec:2}, for a sufficiently small value $\epsilon>0$, let $\mathcal{H}_{\epsilon}^{g}= \mathbb{D}^{2}\setminus \bigcup D_{\epsilon}(w)$ be the hyperbolic multipunctured disk, where the punctures are obtained after removing an open disk $D_{\epsilon}(w)$ of radius $\epsilon$ and centered at $w$, for every $w\in \varphi^{-1}(w_{0})$.  \par

We aim to parameterize the smooth curves in $\mathcal{H}_{\epsilon}^{g}$ whose endpoints are contained in $\partial D_{\epsilon}(\tilde{w}_{0})$ and any other component $C$ of $\partial \mathcal{H}_{\epsilon}^{g}$. Given a curve $\beta$ with these characteristics it can be proven, as in the case $g=1$, that there exists a unique minimal-length curve $\beta_{0}$ in the homotopy class of $\beta$ in $\mathcal{H}_{\epsilon}^{g}$ as curves with endpoints in $\partial D_{\epsilon}(\tilde{w}_{0})$ and $C$. Moreover, Proposition \ref{teo:p1} can be generalized as follows: % This result is very intuitive as it was true in the case $g=1$ from Section \ref{sec:2}, but as far as we know there is not a written proof of this affirmation, and we present it as a conjecture:  
\begin{figure}
  \centering
    \includegraphics[width=12cm]{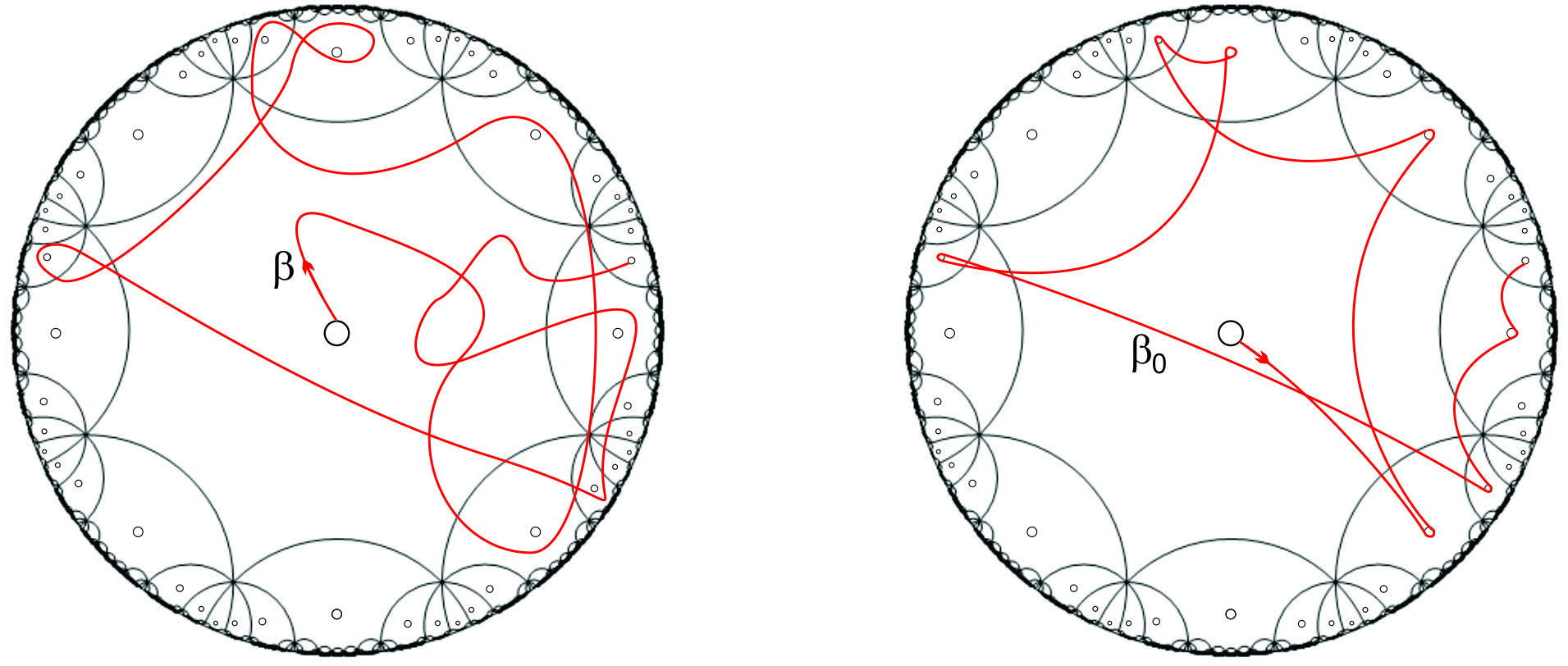}
  \caption{A curve and its minimal-length homotopic curve}
  \label{fig:h3}
\end{figure}
\begin{pr}\label{teo:c1} Let $\beta:[0,1]\rightarrow\mathcal{H}_{\epsilon}^{g}$ be a smooth arc with $\beta(0)\in\partial D_{\epsilon}(\tilde{w}_{0})$ and $\beta(1)\in\partial D_{\epsilon}(\tilde{w}_{0}')$, for some $\tilde{w}_{0}'\in \varphi^{-1}(w_{0})$ (it could be $\tilde{w}_{0}'=\tilde{w}_{0}$). There exists a unique minimal-length curve $\beta_{0}$ within the homotopy class of $\beta$  as curves in $\mathcal{H}_{\epsilon}^{g}$ with endpoints in $\partial D_{\epsilon}(\tilde{w}_{0})$ and $\partial D_{\epsilon}(\tilde{w}_{0}')$. The curve $\beta_{0}$ decomposes as $\beta_{0}=\gamma_{1}\cup\delta_{1}\cup\gamma_{2}\cup\dots\cup\delta_{n}\cup\gamma_{n+1}$, where $\delta_{i}$ is a point in $\partial D_{\epsilon}(\tilde{w}_{i})$ or a monotonous curve around $\partial D_{\epsilon}(\tilde{w}_{i})$, for some $\tilde{w}_{i}\in\varphi^{-1}(w_{0})$, and $\gamma_{i}$ is a geodesic arc sharing endpoints with $\delta_{i}$ and $\delta_{i+1}$ (except for the starting point of $\gamma_{1}$ and the endpoint of $\gamma_{n+1}$).  
\end{pr}

In Figure \ref{fig:h3} we show an arbitrary curve $\beta$ in $\mathcal{H}_{\epsilon}^{2}$ (left-hand picture) and its proposed shortest homotopic curve (right-hand picture). We extend the notions of arc reduction and stabilization of the minimal-length  curve $\beta_{0}$ from Section \ref{sec:2}, namely, there exists a value $\epsilon>0$ such that in the hyperbolic multipunctured disk $\mathcal{H}_{\epsilon}^{g}$, the minimal-length representative $\beta_{0}$ in the class of $\beta$ is simplified (does not admit an arc reduction and it is stabilized). It only remains to assign parameters to the subcurves in the decomposition $\beta_{0}=\gamma_{1}\cup\delta_{1}\cup\gamma_{2}\cup\dots\cup\delta_{n}\cup\gamma_{n+1}$. Given the subcurve $\delta_{i}\subset \partial D_{\epsilon}(w_{i})$ of $\beta_{0}$, it would be possible to define the winding number $m_{i}\in \mathbb{Z}$ of $\delta_{i}$ around $w_{i}$ as in  Preposition \ref{teo:p1}. For $i\in\{1,,2\ldots,n+1\}$, the geodesic arc $\gamma_{i}$ has its endpoints on $\partial D_{\epsilon}(\tilde{w}_{i-1})$ and $\partial D_{\epsilon}(\tilde{w}_{i})$ (set $\tilde{w}_{n+1}=\tilde{w}_{0}'$). Let $\lambda_{i}$ be the oriented geodesic that passes through $\tilde{w}_{i-1}$ and $\tilde{w}_{i}$  (see Figure \ref{fig:h4}). The oriented geodesic $\lambda_{i}$ is described by two ordered points $z_{1},z_{2}\in \partial\mathbb{D}^{2}$. Let $r_{i},s_{i}\in [0,2\pi)$ be the parameters corresponding to the points $z_{1}$ and $z_{2}$, respectively, in the parameterization $f(x)=(cosx,sinx)$, $x\in [0,2\pi)$, of $\mathbb{S}^{1}=\partial\mathbb{D}^{2}$. Note that $\lambda_{i}$ contains infinitely many points of $\varphi^{-1}(w_{0})$, but $\tilde{w}_{i}$ must be the closest point to $\tilde{w}_{i-1}$ in $\lambda_{i}\cap \varphi^{-1}(w_{0})$ in direction of $\lambda_{i}$.

 %Finally, if $l_{i}$ is the hyperbolic distance between $\tilde{w}_{i-1}$ and $\tilde{w_{i}}$, then it is possible to recover the geodesic arc $\gamma_{i}$ by knowing its starting point and the three parameters $(r_{i},s_{i},l_{i})$. 
Under the previous assumptions we propose a parameterization of the family of $(g,1)$-knots:       
 
\begin{figure}
  \centering
    \includegraphics{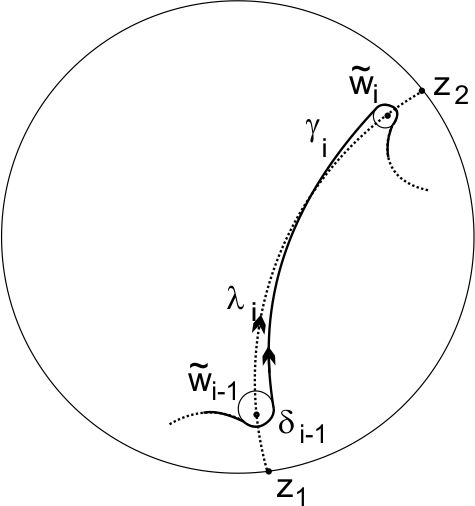}%[width=12cm]{h3}
  \caption{Parameters of a geodesic}
  \label{fig:h4}
\end{figure} 

\begin{teo} \label{teo:c2} Let $K$ be a  $(g,1)$-knot. Then $K$ is parameterized by an ordered sequence of numbers $(r_{1}, s_{1}, m_{1},r_{2},s_{2},m_{2},\ldots,m_{n},r_{n+1},s_{n+1})$, for some $n\geq 0$,  $m_{i}\in \mathbb{Z}$, and $r_{i},s_{i}\in [0,2\pi)$  for every $i$.
\end{teo}

The proof of Theorem \ref{teo:c2} is completely analogous to the poof of Theorem \ref{teo:31} with the obvious adjustments. Suppose $K=A_{0}\cup A_{1}\subset \Sigma^{g}\times I$ is a $(g,1)$-knot in straight bridge position with respect to the standard genus-$g$ surface $\Sigma^{g}$, such that $A_{0}=\{w_{0}\}\times I$. If $\alpha: I\rightarrow \Sigma^{g}\times I$ is a smooth parameterization of $A_{1}$ such that $\alpha(t)\in \Sigma^{g}\times \{t\}$ for each $t\in I$, then $\pi\circ\alpha$ is a closed curve in $\Sigma^{g}$, where $\pi:\Sigma^{g}\times I\rightarrow \Sigma^{g}$ is the projection onto the surface. Let $\tilde{\beta}$ be the lifting of $\pi\circ\alpha$ to $\mathbb{D}^{2}$ and starting at $\tilde{w}_{0}$. Take $\epsilon>0$ sufficiently small and define the hyperbolic multipunctured disk $\mathcal{H}_{\epsilon}^{g}$ as before. Let $\beta$ be the subcurve of $\tilde{\beta}$ that is contained in $\mathcal{H}_{\epsilon}^{g}$, and let $\beta_{0}$ be its minimal-length homotopic curve in $\mathcal{H}_{\epsilon}^{g}$ under a homotopy $H$. The homotopy $H$ induces an isotopy in $K$ as in Theorem \ref{teo:31} and the parameterization of $\beta_{0}$ induces the parameterization of $K$.

\vspace{.8cm}

 \noindent\textbf{Acknowledgment}

This research work was supported by project FORDECYT 265667 and CONACYT Postdoctoral Fellowship. The author is grateful to professors M. Neumann-Coto, M. Eudave-Mu\~noz and J.C. G\'omez-Larra\~naga for their valuable observations and comments.

\vspace{.5cm}
\noindent CENTRO DE INVESTIGACI\'ON EN MATEM\'ATICAS, A.C, JALISO S/N, COL. VALENCIANA, CP: 36023, GUANAJUATO, GTO., M\'EXICO.\\
Email-address: frias4@cimat.mx

\end{document}